\documentclass[12pt,a4paper]{amsart}
\usepackage[utf8]{inputenc}
\usepackage[english]{babel}
\usepackage[T1]{fontenc}
\usepackage{amsmath}
\usepackage{mathtools}
\usepackage{amsfonts}
\usepackage{amssymb}
\usepackage{biblatex}
\usepackage{vmargin}
\usepackage{setspace}
\author{Léo Morin}
\title[]{Spectral asymptotics for the semiclassical Bochner Laplacian of a line bundle with constant rank curvature}
\makeatletter

\@addtoreset{equation}{section}
\makeatother

\newcounter{item}

\newtheorem{theorem}[item]{\sffamily Theorem}

\newtheorem{lemma}[item]{\sffamily Lemma}
\newtheorem{corollary}[item]{\sffamily Corollary}

\newtheorem*{theorem*}{\sffamily Theorem}
\newtheorem*{definition*}{\sffamily Definition}
\newtheorem*{proposition*}{\sffamily Proposition}
\newtheorem*{lemma*}{\sffamily Lemma}
\newtheorem*{corollary*}{\sffamily Corollary}
\newtheorem*{assumption*}{\sffamily Assumptions}

\usepackage[explicit]{titlesec}
\titleformat{\section}{\centering\Large\bfseries}{\thesection \ --}{0.7em}{\Large\bfseries #1}
\titleformat{\subsection}{\centering\large\bfseries}{\thesubsection \ --}{0.4em}{\large\bfseries #1}
\titleformat{\subsubsection}{\centering\bfseries}{\thesubsubsection \ --}{0.4em}{\bfseries #1}

\let\emph\relax

\newcommand{\dd}{\mathrm{d}}
\newcommand{\R}{\mathbf{R}}
\newcommand{\C}{\mathbf{C}}
\newcommand{\N}{\mathbf{N}}
\newcommand{\Z}{\mathbf{Z}}
\newcommand{\B}{\mathbf{B}}

\newcommand{\Cinf}{\mathcal{C}^{\infty}}
\newcommand{\Ld}{\mathsf{L}^2}
\newcommand{\Dom}{\mathsf{Dom}}

\newcommand{\supp}{\mathsf{supp}}

\newcommand{\Lpj}{\mathcal{L}_{p}^{(j)}}
\newcommand{\Lhj}{\mathcal{L}_{h}^{(j)}}

\newcommand{\grandO}{\mathcal{O}}

\begin{document}

\maketitle

\begin{abstract}
The goal of this paper is manyfold. Firstly, we want to give a short introduction to the Bochner Laplacian and explain why it acts locally as a magnetic Laplacian. Secondly, given a confining magnetic field, we use Agmon-like estimates to reduce its spectral study to magnetic Laplacians, in the semiclassical limit. Finally, we use this to translate already-known spectral asymptotics for the magnetic Laplacian to the Bochner Laplacian.
\end{abstract}

\section{Introduction}

\subsection{Motivations and context}

The spectral theory of the magnetic Laplacian, and the Bochner Laplacian, has given rise to many interesting questions. First motivated by the Ginzburg-Landau theory, bound states of the magnetic Laplacian $(ih\dd + A)^* (ih \dd + A)$ on a Riemannian manifold in the semiclassical limit $h \rightarrow 0$ were studied in many works (see the books \cite{FournaisHelffer,Raymond}), and appears to have very various behaviours according to the variations of the magnetic field $B = \dd A$ and the boundary conditions. The first main technique consisted in the construction of approximated eigenfunctions (see for instance the works of Helffer-Mohamed \cite{HelMo96} and Helffer-Kordyukov \cite{HelKor14,HelKor12,HelKor11}). More recently, an other approach was developped, which consists in an approximation of the operator itself, using semiclassical tools such as microlocalisation estimates and Birkhoff normal forms (As in Raymond-Vu Ngoc \cite{Birkhoff2D} and Helffer-Kordyukov-Raymond-Vu Ngoc \cite{Birkhoff3D}). In the semiclassical limit $h \rightarrow 0$, we recover the classical behaviour of a particle exposed to the magnetic field $B$, since the magnetic Laplacian is the quantification of the classical energy.\\

If we are given a magnetic field $B$ which is not exact, there is no potential $A$ and we cannot define the magnetic Laplacian. However, the Bochner Laplacian $\frac{1}{p^2}\Delta^{L^p}$ appears to be the suitable generalization in this case, since it acts \textit{locally} as a magnetic Laplacian. In this context the semiclassical parameter is $p = h^{-1}$. Its spectral theory appears to be deeply related to holomorphic structures and to the Kodaira Laplacian (or the renormalized Bochner Laplacian more generaly). For instance, this is exploited in the works of Marinescu-Savale \cite{MaSa18}, and Kordyukov \cite{Kor19}. In this last paper, the case of non-degenerate magnetic wells with full-rank magnetic field is studied, and expansions of the ground states energies are given, using quasimodes.  In \cite{MagBNF}, similar results were obtained in the special case of the magnetic Laplacian, using a Birkhoff normal form, but also giving a description of semi-excited states, and a Weyl law. In \cite{Morin-PresymplecticBNF}, these result are generalized to constant-rank magnetic fields.\\

The spectral theory of the Bochner Laplacian is also deeply related to the global geometry of complex manifolds. For instance, in \cite{Dem85}, a Weyl law was proven for $\Delta^{L^p}$, and used to get Morse inequalities and Riemann-Roch formulas.\\

In this paper, we use Agmon-like estimates to reduce the spectral theory of the Bochner Laplacian with magnetic wells to magnetic Laplacians. Then, we deduce spectral asymptotics for the Bochner Laplacian using the results of \cite{MagBNF,Morin-PresymplecticBNF}.

\subsection{The Bochner Laplacian on a line bundle}

Let $(M,g)$ be a compact oriented manifold of dimension $d > 1$. We consider a complex line bundle $L \rightarrow M$ over $M$, endowed with a Hermitian metric $h$. In other words, we associate to each $x \in M$ a $1$-dimensional complex vector space $L_x$, and a Hermitian product $h_x$ on $L_x$. $L$ is a $d+1$-dimensional manifold such that $L = \bigcup_{x \in M} L_x$. A smooth section of $L$ (or $L$-valued function) is a smooth function $s : M \rightarrow L$ such that $s(x) \in L_x$. It is the generalisation of the notion of function $f : M \rightarrow \C$, but here the target space can vary with $x \in M$. Similarily, $L$-valued $k$-forms are sections of $\wedge^k T^*M \otimes L$. We denote by $\Cinf(M,L)$ the set of smooth sections of $L$, and $\Omega^k(M,L)$ the set of smooth $L$-valued $k$-forms.\\

We take $\nabla^L$ a Hermitian connexion on $(L,h)$. It is the generalisation of the exterior derivative $\dd$. The underlying idea is that the "derivative" of a $L$-valued function should be $L$-valued too. $\nabla^L : \Omega^k(M,L) \rightarrow \Omega^{k+1}(M,L)$ satisfies:
\begin{equation}
\nabla^L ( s \alpha ) = \nabla^L s \wedge \alpha + s \dd \alpha, \quad \forall s \in \Cinf(M,L), \quad \alpha \in \Omega^k(M,\C),
\end{equation}
\begin{equation}
\dd h (s_1,s_2) = h (\nabla^L s_1, s_2) + h( s_1, \nabla^L s_2), \quad \forall s_1,s_2 \in \Cinf(M,L).
\end{equation}
One can prove that $(\nabla^L)^2 : \Omega^0(M,L) \rightarrow \Omega^2(M,L)$ acts as a multiplication. There exists a real closed $2$-form $B$ on $M$ such that:
\begin{equation}
(\nabla^L)^2 s = i B s, \quad \forall s \in \Cinf(M,L).
\end{equation}

\textit{Example : The trivial line bundle.} The line bundle $L = M \times \C$, such that $L_x = \lbrace x \rbrace \times \C$ is called the trivial line bundle. We identify sections $s\in \Cinf(M,L)$ with functions $f \in \Cinf(M)$ by $s(x) = (x,f(x))$. Similarily, $L$-valued $k$-forms are identified with $\C$-valued $k$-forms, and we recover the usual differential objects on $M$. If $L$ is endowed with the Hermitian product $h_x((x,z_1),(x,z_2)) = z_1 \overline{z}_2$, we call $(L,h)$ the trivial Hermitian line bundle. We write $h(z_1,z_2)$ for short. Hermitian connexions on the trivial line bundle are given by $\nabla_{\alpha} = \dd + i \alpha$ where $\alpha \in \Omega^1(M, \R)$ and $\dd$ is the exterior derivative. The curvature of $\nabla_{\alpha}$ is $\nabla_{\alpha}^2 = i \dd \alpha$, as shown by the easy but enlightening calculation:
\begin{equation}
\begin{matrix*}[l]
\nabla_{\alpha}^2 f &= (\dd  + i  \alpha )(\dd f  + i f \alpha) = \dd^2 f + i \alpha \wedge \dd f + i \dd (f \alpha) + i f \alpha \wedge \alpha \\
&= i \alpha \wedge \dd f + i \dd f \wedge \alpha + i f \dd \alpha = i f \dd \alpha.
\end{matrix*}
\end{equation}

Let us describe now the Bochner Laplacian $\Delta^L$ associated to a Hermitian connexion $\nabla^L$ on a Hermitian complex line bundle $(L,h)$. First note that the spaces $\Cinf(M,L) = \Omega^0(M,L)$ and $\Omega^1(M,L)$ are endowed with $\Ld$-norms. The norm of a section $s \in \Cinf(M,L)$ is:
\begin{equation}
\Vert s \Vert^2 = \int_M h_x(s(x),s(x)) \dd \nu_g(x),
\end{equation}
where $\dd \nu_g$ denotes the volume form of the oriented Riemannian manifold $(M,g)$. We denote by $\Ld(M,L)$ the completion of $\Cinf(M,L)$ for this norm. The definition of the norm of a $L$-valued $1$-form $\alpha$ is a little more involved. First, using a partition of unity, it is enough to define it for $\alpha \in \Omega^1(U,L)$ where $U$ is a small open subset of $M$. If $U$ is small enough, there exists a section $e \in \Cinf(U,L)$ such that $h_x(e(x),e(x)) = 1$. Then for any $\alpha \in \Omega^1(U,L)$, there exists a unique $X \in TM$ such that $\alpha_x( \bullet) = g_x(X_x, \bullet) e_x$ (we identify $1$-forms with tangent vectors using the metric $g$). We define:
\begin{equation}
\Vert \alpha \Vert^2 = \int_M g_x(X_x,X_x) \dd \nu_g(x).
\end{equation}
The completion of $\Omega^1(M,L)$ for this norm is denoted by $\Ld \Omega^1(M,L)$: it is the space of square-integrable $L$-valued $1$-forms. These norms are associated with scalar products, denoted by brackets $\langle .,. \rangle$.\\

The formal adjoint of $\nabla^L : \Omega^0(M,L) \rightarrow \Omega^1(M,L)$ for these scalar products is denoted by $(\nabla^L)^* : \Omega^1(M,L) \rightarrow \Omega^0(M,L)$. The Bochner Laplacian $\Delta^L$ is the self-adjoint extension of $(\nabla^L)^* \nabla^L$. It is the operator associated with the quadratic form:
\begin{equation}
Q(s_1,s_2) = \langle \nabla^L s_1, \nabla^L s_2 \rangle.
\end{equation}
We denote by $\Dom ( \Delta^L)$ its domain. $\Cinf(M,L)$ is a dense subspace of $\Dom( \Delta^L)$ and:
\begin{equation}
\langle \Delta^L s_1, s_2 \rangle = \langle \nabla^L s_1, \nabla^L s_2 \rangle, \quad \forall s_1, s_2 \in \Dom( \Delta^L).
\end{equation}
Since $M$ is compact, one can prove that $\Delta^L$ has compact resolvent, and we denote by
\begin{equation}
\lambda_1(\Delta^L) \leq \lambda_2 ( \Delta^L) \leq ...
\end{equation}
the non-decreasing sequence of its eigenvalues. We will use the following notation for the Weil counting function:
$$N(\Delta^{L}, \lambda) := \sharp \lbrace j; \lambda_j (\Delta^{L}) \leq \lambda \rbrace.$$

In this paper, we are interested in the semiclassical limit, i.e. the high curvature limit "$B \rightarrow + \infty$". We can increase the curvature $B$ using tensor products of $L$. For any $p \in \N$, we denote by $L^p = L \otimes ... \otimes L$ the $p$-th tensor power of $L$. $L^p$ is still a complex line bundle of $M$, with $L^p_x = L_x \otimes ... \otimes L_x$. It is endowed with the Hermitian product $h^p_x ( s_1 \otimes ... \otimes s_p, s_1 \otimes ... \otimes s_p) = \Pi_{i=1}^p h_x(s_i,s_i)$. The connexion $\nabla^L$ induces a Hermitian connexion $\nabla^{L^p}$ on $L^p$ by the formula:
$$\nabla^{L^p} ( s_1 \otimes ... \otimes s_p) = (\nabla^L s_1) \otimes ... \otimes s_p + ... + s_1 \otimes ... \otimes (\nabla^L s_p).$$
The curvature of $\nabla^{L^p}$ is 
\begin{equation}
(\nabla^{L^p})^2 = i p B.
\end{equation}
Hence, the high curvature limit is $p \rightarrow + \infty$. We want to invastigate the behaviour of $\lambda_j(\Delta^{L^p})$ and the corresponding eigensections in the limit $p \rightarrow + \infty$.

\subsection{Main results}

One can measure the "intensity" of the curvature $B$ (the magnetic field) in the following way. We denote by $\B_x : T_xM \rightarrow T_xM$ the linear operator defined by
\begin{equation}
g_x(\B_x U, V) = B_x(U,V), \quad \forall U, V \in T_xM.
\end{equation}
$\B_x$ is real and skew-symmetric with respect to $g_x$, and thus its eigenvalues lie on the imaginary axis and are symmetric with respect to the origin. We denote its non-zero eigenvalues by:
\begin{equation}
\pm i \beta_1(x), \cdots , \pm i \beta_s(x),
\end{equation}
with $\beta_j(x) >0$. Hence, the rank of $\B_x$ is $2s$ and might depend on $x$, but we will soon assume that $s$ is constant, at least locally. The magnetic intensity is the function $b : M \rightarrow \R_+$ defined by:
\begin{equation}
b(x) = \sum_{j=1}^{s(x)} \beta_j(x).
\end{equation}
This function is continuous on $M$, but not smooth in general. However, note that it is smooth on a neighborhood of any point $x_0$ where the $(\beta_j(x_0))_{1 \leq j \leq s}$ are simple (if $s$ is locally constant near $x_0$).\\

One of the purposes of this article is to show that the eigensections of $\Delta^{L^p}$ are localized near the minimum points of $b$, and to deduce that the low-lying spectrum of $\Delta^{L^p}$ is given by magnetic Laplacians on neighborhoods of the minimal points of $b$.\\

We will do the following assumptions.
\begin{assumption*}
\begin{enumerate}
\item[(A1)] The minimal value of $b$ is only reached at non degenerate points $x_1, \cdots, x_N \in M$. We denote by $b_0 = b(x_j) = \min_{x \in M} b$. \\
\item[(A2)] The rank of $\B$ is constant on small neighborhoods $U_1, \cdots , U_N$ of $x_1, \cdots,  x_N$. We denote by $2 s_j$ the rank of $\B_{x_j}$.\\
\item[(A3)] We assume $b_0 > 0$, which is equivalent to say that $s_j > 0$ for any $j$.
\end{enumerate}
\end{assumption*}

As noticed in several papers, one can prove using Agmon-like estimates that the eigensections of $\Delta^{L^p}$ associated to low-lying eigenvalues are exponentially localized near $\lbrace x \in M, \quad b(x) = b_0 \rbrace,$ in the limit $p \rightarrow + \infty$. Now let us present the local model operators on $U_j$.\\

Recall that the $2$-form $B$ is closed: $\dd B = 0$. Hence, if the open sets $U_j$ are small enough, $B$ is exact on $U_j$: there exists $A_j \in \Omega^1(U_j)$ such that $B = \dd A_j$ on $U_j$. We denote by $\Lpj$ the Dirichlet realization of $(d + i p A_j)^* (d + i p A_j)$ on $\Ld(U_j)$. It is the self-adjoint operator associated to the following sesquilinear form on $\Cinf_0(U_j)$:
\begin{equation}
Q_j(u,v) = \int_M (\dd u + i p A_j u) \overline{(\dd v + i p A_j v)} \dd \nu_g.
\end{equation}

We prove the following Theorem.

\begin{theorem}\label{MainTheorem}
Let $\alpha \in (0,1/2)$. Under assumptions $(A1)$ and $(A3)$, if $\eta, \varepsilon >0$ are small enough, then:
\begin{equation}
\lambda_k(\Delta^{L^p}) = \lambda_k \left( \mathcal{L}_p^{(1)} \oplus ... \oplus \mathcal{L}_p^{(N)} \right) + \grandO(exp(-\varepsilon p^{\alpha})),
\end{equation}
uniformly with respect to $k \in [1,K_p]$, where 
$$K_p = \min \left( N(\Delta^{L^p}, (b_0 + \eta)p) , N(\mathcal{L}_p^{(1)} \oplus ... \oplus \mathcal{L}_p^{(N)} , (b_0 + \eta)p) \right),$$
and $N(\mathcal{A}, \lambda)$ denotes the number of eigenvalues of an operator $\mathcal{A}$ below $\lambda$, counted with multiplicities.
\end{theorem}

As a corollary, we can deduce spectral asymptotics for $\Delta^{L^p}$ from already-known results for $\Lpj$. Let us recall some of these results here.\\

\subsubsection{The full-rank case}
 
Under the assumptions $(A1)-(A2)-(A3)$, we fix a $j \in \lbrace 1 , \cdots , N\rbrace$, and we denote by $B_j = \dd A_j$. Hence, $B_j$ is just the restriction of $B$ to the small open set $U_j$, where it admits a primitive $A_j$. $\Lhj$ is the magnetic Laplacian with Dirichlet boundary conditions on $U_j$, with magnetic field $B_j$. We first focus on the full-rank case, when the rank of $B_j$ is maximal: $2s_j = d$. We define $r_j \in \N$ by the condition
\begin{equation}\label{resonances1}
\forall n \in \Z^{s_j}, \quad 0 < \sum_{\ell = 1}^{s_j} \vert n_{\ell} \vert < r_j \Rightarrow \sum_{\ell=1}^{s_j} n_{\ell} \beta_{\ell}(x_j) \neq 0.
\end{equation}
Note that, if the $\beta_{\ell}(x_j)$ are pairwise distinct, we can choose $r_j \geq 3$. Moreover, if the open set $U_j$ is small enough we have, for all $x \in U_j$ and $n \in \Z^{s_j}$,
\begin{equation}
0 < \sum_{\ell = 1}^{s_j} \vert n _{\ell} \vert < r_j \implies \sum_{\ell=1}^{s_j} n_{\ell} \beta_{\ell}(x) \neq 0.
\end{equation}
The following Theorem was proved in \cite{MagBNF}.

\begin{theorem}
We assume $(A1)-(A2)-(A3)$ with $2s_j =d$ and $r \geq 3$ in \eqref{resonances1}. Let $\eta, \varepsilon >0$ small enough. Then there exists a symplectomorphism $\psi : U_j \rightarrow T^* \R^{d/2}$ such that:
\begin{equation}
\frac{1}{p^2}\lambda_k( \Lpj ) = \lambda_k \left( \bigoplus_{n \in \N^d} \mathcal{N}_p^{[j,n]} \right) + \grandO (p^{-r_j / 2 + \varepsilon}),
\end{equation}
uniformly with respect to $k \in [1, \tilde{K}_p]$, where $\mathcal{N}_p^{[j,n]}$ is a pseudo-differential operator with principal symbol:
$$\sigma ( \mathcal{N}_p^{[j,n]} ) = \frac{1}{p} \sum_{\ell = 1}^{s_j}  (2n_{\ell} + 1) \beta_{\ell} \circ \psi^{-1}(x,\xi),$$
and
$$\tilde{K}_p = \min \left( N(\Lpj , (b_0 + \eta)p)  , N( \oplus_n \mathcal{N}_p^{[j,n]} , (b_0 + \eta)p^{-1} ) \right).$$
\end{theorem}

Hence, we have a description of the semi-excited states of $\Lpj$. In the same paper, Weyl estimates are proven for $\Lpj$. We directly deduce the following Weyl estimates for $\Delta^{L^p}$. A Similar formula was proven in \cite{Dem85} using a local approximation of the magnetic field by a constant field.

\begin{corollary}
Assume $(A_1)-(A_2)-(A_3)$, and for any $j \in \lbrace 1, \cdots , N \rbrace$ that $s_j = d/2$, and that $(\beta_{\ell}(x_j))_{1 \leq \ell \leq N}$ are pairwise distinct. Then, for $\eta >0$ small enough,
\begin{equation}
N(\Delta^{L^p}, (b_0 + \eta)p) \sim \left( \frac{p}{2 \pi } \right) ^{d/2} \sum_{n \in \N^{d/2}}\int_{b^{[n]}(x) \leq b_0 + \eta} \frac{B^{d/2}}{(d/2)!},
\end{equation}
in the limit $p \rightarrow + \infty$, where $b^{[n]}(x) = \sum_{\ell=1}^{d/2} n_{\ell} \beta_{\ell}(x)$.
\end{corollary}

Finally, we also deduce asymptotic expansions of the first eigenvalues.

\begin{corollary}
Assume $(A1)-(A2)-(A3)$, and for any $j \in \lbrace 1, \cdots N \rbrace$ that $s_j = d/2$ and that $(\beta_{\ell}(x_j))_{1 \leq \ell \leq N}$ are pairwise distinct, and $r := \min_j r_j \geq 5$. Then, for any $k \in \N$ and $\varepsilon >0$,
$$\lambda_k(\Delta^{L^p}) = b_0 p + \sum_{i=0}^{r-5} \alpha_{i,k} p^{-i/2} + \grandO( p^{2-r/2 + \varepsilon}),$$
for some coefficients $\alpha_{i,k} \in \R$.
\end{corollary}

\textit{Remark.} We also have geometric interpretations of the coefficients. First, the full expansion comes from the effective operator $\mathcal{N}_p^{[j,0]}$, which is the reduction of $\Lpj$ to the lowest energy of the Harmonic oscillator describing the classical cyclotron motion. Moreover, $\alpha_{0,k}$ is given by  an eigenvalue of an other Harmonic oscillator whose symbol is the Hessian of $b$ at $x_j$ (for some $1 \leq j \leq N$): it describes a slow drift of the classical particle arround $x_j$. If the eigenvalues of this oscillator are simple, then a Birkhoff normal form can be used to show that $\alpha_{i,k} = 0$ if $i$ is odd.

\subsubsection{The constant-rank case}

In the non-full-rank case, the kernel of $B$ (which corresponds to the directions of the field lines), has a great influence on the spectrum of $\Delta^{L^p}$. Fix $1 \leq j \leq N$. If the rank of $B_j$ is constant, equal to $2s_j$, then its kernel as dimension $k_j = d - 2s_j$. The partial Hessian of $b$ at $x_j$, in the directions of the Kernel of $B_j$, is non-degenerate. we denote by
\begin{equation}
\nu_{j,1}^2, \cdots, \nu_{j,k_j}^2
\end{equation}
its eigenvalues. For simplicity, we will make the following non-resonance assumptions (however, we can deal with resonances using a resonance order $r$ as in the full-rank case).

\begin{assumption*}
(A4) For every $j$, $(\beta_{\ell}(x_j))_{1 \leq \ell \leq s_j}$ are non-resonnant in the following sense:
$$ \forall n \in \Z^{s_j}, \quad n \neq 0 \implies  \sum_{\ell = 1}^{s_j} n_{\ell} \beta_{\ell}(x_j) \neq 0.$$
(A5) For every $j$ such that $k_j >0$, $(\nu_{j,\ell})_{1 \leq \ell \leq k_j}$ are non resonnant in the following sense:
$$ \forall n \in \Z^{k_j}, \quad n \neq 0 \implies  \sum_{\ell = 1}^{k_j} n_{\ell} \nu_{j,\ell} \neq 0.$$
\end{assumption*}

Applying the results of \cite{Morin-PresymplecticBNF} to get spectral asymptotics for $\Lhj$, we deduce from Theorem \ref{MainTheorem} the following corollary.

\begin{corollary}
Assume $(A1)-(A2)-(A3)-(A4)-(A5)$, and let $n \in \N$. Then $\lambda_n(\Delta^{L^p})$ admits a full asymptotic expansion in powers of $p^{-1/2}$:
$$\lambda_n(\Delta^{L^p}) = b_0 p + \kappa p^{1/2} + \sum_{i \geq 0} \alpha_{i,n} p^{-i/2} + \grandO(p^{-\infty}).$$
Moreover:
\begin{enumerate}
\item[•] If there is at least one $j$ such that $k_j = 0$, then $\kappa = 0$.\\
\item[•] If $\forall j \in \lbrace 1 , \cdots , N \rbrace, k_j >0$, then $\kappa = \min_{j=1, \cdots, N} \sum_{\ell=1}^{k_j} \nu_{j,\ell}.$
\end{enumerate}
\end{corollary}

\section{Some Remarks}

\subsection{The Bochner Laplacian and the magnetic Laplacian are locally the same}

If $U$ is any open subset of $M$ such that there exists a non-vanishing section $e \in \Cinf(U,L)$, then any $s \in \Cinf(U,L)$ can be written $s = u e$ for some $u \in \Cinf(M)$. Hence, $$\nabla s = (\nabla e) u  + e (\dd u)=  e [(\dd + i A)u],$$
with $\nabla e = e iA$. Moreover, 
$$
\begin{matrix*}[l]
\nabla^2 s &= \nabla e \wedge [(\dd + iA)u] + e \dd [(\dd + i A)u] \\
 &= e (iA \wedge \dd u) + e (iA \wedge iA) u + e \dd^2 u + i e u \dd A + e \dd u \wedge iA\\
 &= i e u \dd A = (i \dd A) s,
\end{matrix*}
$$
and thus $B = \dd A$. Hence $\nabla$ acts locally as $\dd + i A$, and $\Delta^L$ as the magnetic Laplacian $(\dd + i A)^* (\dd + i A)$. This is the core of Theorem \ref{MainTheorem}.

\subsection{On the quantization of a magnetic field}

If we are given a closed $2$-form $B$ (the magnetic field), the quantization question constist in finding a quantum operator associated to $B$. If $B$ is exact, this question is answered by the semiclassical magnetic Laplacian $(\hbar \dd + i A)^* (\hbar \dd + i A)$, with $B = \dd A$. Here, $\hbar>0$ is the semiclassical parameter (Planck's constant) and the semiclassical limit is $\hbar \rightarrow 0$.\\

If $B$ is not exact, but if there exists an Hermitian line bundle with Hermitian connexion such that $\nabla^2 = i B$, then the Bochner Laplacian $\nabla^* \nabla$ acts locally as the magnetic Laplacian and hence it is a good candidate. Moreover, we have \textit{locally} $$\Delta^{L^{p}} = (\dd + i p A)^{*}(\dd +i pA) = p^{2} (\frac{1}{p} \dd + i A)^{*}(\frac{1}{p} \dd + i A),$$ 
so that the semiclassical parameter is now $\hbar = \frac{1}{p}$ (Also notice the $p^{2}$ factor which is important for the eigenvalue asymptotics). The limit $\hbar \rightarrow 0$ is equivalent to $p \rightarrow + \infty$ exept that the semiclassical parameter becomes discrete ($p \in \N$).\\

A new question arises : When does such an Hermitian line bundle exists ? Weil's Theorem states that it exists if and only if $B$ satisfies the prequantization condition:
\begin{equation}
[B] \in 2\pi \Z,
\end{equation}
where $[B]$ denotes the cohomology class of $B$. This condition also enlightens the discreteness of the semiclassical parameter. Indeed, if one wants to quantize the magnetic field $\frac{1}{\hbar}B$, then one must have $\left[ \frac{1}{\hbar} B \right] \in 2 \pi \Z$, and thus $\frac{1}{p} \in \Z$, unless $[B]=0$ which means that $B$ is exact (and thus we can use the magnetic Laplacian !).

\section{Proof of Theorem \ref{MainTheorem}}

\subsection{Agmon-like estimates}

In this section, we prove exponential decay on the eigensections of $\Delta^{L^p}$, away from the set $\lbrace x_1, \cdots x_N \rbrace$. We need the following Lemma.

\begin{lemma}\label{Lemma01-minoration}
There exist $p_0>0$ and $C_0>0$ such that, for $p \geq p_0$ and $s \in \Cinf(M,L)$,
$$(1+ \frac{C_0}{p^{1/4}}) \Vert \nabla^{L^p} s \Vert^2 \geq p \int_M (b(x) - \frac{C_0}{p^{1/4}} ) \vert s(x) \vert^2 \dd x.$$
\end{lemma}

\begin{proof}
Take a partition of unity $(\chi_{\alpha})_{\alpha}$ on $M$, such that $\supp \chi_{\alpha} \subset U_{\alpha}$ with $U_{\alpha}$ a small open subset of $M$, and
\begin{equation}\label{eq00-Lemma01}
1 = \sum_{\alpha} \chi_{\alpha}^2, \quad \sum_{\alpha} \vert \dd \chi_{\alpha} \vert^2 \leq C.
\end{equation}
Then, for any $s \in \Cinf(M,L)$ we have 
\begin{equation}\label{eq01-Lemma01}
\Vert \nabla^{L^p} s \Vert^2 = \sum_{\alpha} \Vert \nabla^{L^p} (\chi_{\alpha} s) \Vert^2 - \sum_{\alpha} \Vert \vert \dd \chi_{\alpha} \vert s \Vert^2,
\end{equation}
and $\chi_{\alpha} s \in \Cinf_0(U_{\alpha}, L)$. If every $U_{\alpha}$ are small enough, we can find a non-vanishing section $e_{\alpha}$ on $U_{\alpha}$ which we can use to trivialize the line bundle. Writting $\chi_{\alpha} s = u_{\alpha} e_{\alpha}$ for some $u_{\alpha} \in \Cinf_0(U_{\alpha})$, we have
\begin{equation}\label{eq02-Lemma01}
\nabla^{L^p} (\chi_{\alpha} s) = \left[  (\dd + i p A_{\alpha}) u_{\alpha} \right] e_{\alpha},
\end{equation} 
where $A_{\alpha} \in \Omega^1(U_{\alpha})$ is such that $B = \dd A_{\alpha}$. For the magnetic Laplacian $(\dd+ipA_{\alpha})^*(\dd + i p A_{\alpha})$, the desired inequality is well-known: There exist $p_{\alpha} \in \N$, $C_{\alpha} >0$ such that, for every $u_{\alpha} \in \Cinf(U_{\alpha})$, and $p \geq p_{\alpha}$:
\begin{equation}\label{eq03-Lemma01}
 (1+ \frac{C_{\alpha}}{p^{1/4}}) \Vert (\dd + ip A_{\alpha}) u_{\alpha} \Vert^2 \geq p \int_{U_{\alpha}} (b(x) - \frac{C_{\alpha}}{p^{1/4}}) \vert u_{\alpha}(x) \vert^2 \dd x.
\end{equation}
Using \eqref{eq01-Lemma01}, \eqref{eq02-Lemma01}, and \eqref{eq03-Lemma01}, we deduce that
\begin{equation}
(1+ \frac{C_0}{p^{1/4}}) \Vert \nabla^{L^p} s \Vert^2 \geq p \int_M (b(x) - \frac{C_0}{p^{1/4}}) \vert s(x) \vert^2 \dd x - (1+ \frac{C_0}{p^{1/4}}) \sum_{\alpha} \Vert \vert \dd \chi_{\alpha} \vert s \Vert^2,
\end{equation}
with $C_0 = \max_{\alpha} C_{\alpha}$, and for $p \geq \max_{\alpha} p_{\alpha}$. Finally, \eqref{eq00-Lemma01} yields
\begin{equation}
(1+ \frac{C_0}{p^{1/4}}) \sum_{\alpha} \Vert \vert \dd \chi_{\alpha} \vert s \Vert^2 \leq C (1 + \frac{C_0}{p^{1/4}} ) \Vert s \Vert^2 \leq \tilde{C} p^{3/4} \Vert s \Vert^2.
\end{equation}
Hence, up to changing $C_0$ into $C_0 + \tilde{C}$, Lemma \ref{Lemma01-minoration} is proved.
\end{proof}

Now we can use Lemma \ref{Lemma01-minoration} to prove Agmon-like decay estimates.

\begin{theorem}\label{Thm-Agmon}
Let $\alpha \in (0,1/2)$, $\eta >0$, and $K_{\eta} = \lbrace b(x) \leq b_0 + 2\eta \rbrace$. There exist $C>0$ and $p_0 >0$ such that, for all $p \geq p_0$ and all eigenpair $(\lambda, \psi)$ of $\Delta^{L^p}$ with $\lambda \leq (b_0 + \eta) p$,
$$\int_M \vert e^{\dd(x,K_{\eta})p^{\alpha}} \psi \vert^2 \dd q \leq C \Vert \psi \Vert^2.$$
\end{theorem}

\begin{proof}
Let $\Phi : M \rightarrow \R$ be a Lipschitz function. The Agmon formula is:
\begin{equation}\label{eq-AgmonFormula}
\langle \Delta^{L^p} e^{\Phi} \psi, e^{\Phi} \psi \rangle = \lambda \Vert e^{\Phi} \psi \Vert^2 + \Vert \dd \Phi e^{\Phi} \psi \Vert^2.
\end{equation}
Using Lemma \ref{Lemma01-minoration}, we deduce that:
\begin{align*}
\int ( pb(x) - C_0p^{3/4} - (1+ C_0 p^{-1/4})(\lambda +\vert \dd \Phi \vert^2)  ) \vert e^{\Phi} \psi \vert^2 \dd x\leq 0.
\end{align*}
We split this integral into two parts.
\begin{align*}
\int_{K^c_{\eta}} & ( pb(x) - C_0p^{3/4} - (1+ C_0 p^{-1/4})(\lambda +\vert \dd \Phi \vert^2)  ) \vert e^{\Phi} \psi \vert^2 \dd x \\
& \leq \int_{K_{\eta}} ( -pb(x) + C_0p^{3/4} + (1+ C_0 p^{-1/4})(\lambda +\vert \dd \Phi \vert^2)  ) \vert e^{\Phi} \psi \vert^2 \dd x
\end{align*}
We choose $\Phi$:
$$\Phi_m(x) = \chi_m(d(x,K_{\eta}))p^{\alpha}, \quad \text{for } m>0,$$
where $\chi_m(t) = t$ for $t<m$, $\chi_m(t) =0$ for $t>2m$, and $\chi_m'$ uniformly bounded with respect to $m$. Since $\Phi_m=0$ on $K_{\eta}$ and $ p b(x) - C_0 p^{3/4} > 0$ for $p$ large enough, we have:
\begin{align*}
\int_{K^c_{\eta}} & ( pb(x) - C_0p^{3/4} - (1+ C_0 p^{-1/4})(\lambda +\vert \dd \Phi_m \vert^2)  ) \vert e^{\Phi_m} \psi \vert^2 \dd x \\
& \leq (b_0 + \eta) p \int_{K_{\eta}} (1+ C_0 p^{-1/4}) \vert \psi \vert^2 \dd x \leq C p \Vert \psi \Vert^2.
\end{align*}
Moreover, since $\lambda \leq (b_0 + \eta) p$ and $\vert \dd \Phi_m \vert^2 \leq C p^{2 \alpha}$,
\begin{align*}
\int_{K^c_{\eta}} (p b(x) - C_0 p^{3/4} - (1+ C_0 p^{-1/4})(b_0 p + \eta p + Cp^{2\alpha}) \vert e^{\Phi_m} \psi \vert^2 \dd x &\leq C p \Vert \psi \Vert^2\\
p \int_{K^c_{\eta}} (b(x) - (b_0 + \eta) - C_0p^{-1/4} - \tilde{C}p^{2\alpha -1}) \vert e^{\Phi_m} \psi \vert^2 \dd x &\leq C p \Vert \psi \Vert^2,
\end{align*}
for $p$ large enough. But $b(x) > b_0 + 2 \eta$ on $K_{\eta}^c$, so there is a $\delta >0$ and $p_0>0$ such that,  for $p\geq p_0$:
\begin{align*}
\delta \int_{K^c} \vert e^{\Phi_m} \psi \vert^2 \dd q \leq C \Vert \psi \Vert^2.
\end{align*}
Since $\Phi_m = 0$ on $K$, we get a new $C>0$ such that:
$$\Vert e^{\Phi_m} \psi \Vert^2 \leq C \Vert \psi \Vert^2,$$
and we can use Fatou's lemma in the limit $m \rightarrow + \infty$ to get the desired inequality.
\end{proof}

\begin{corollary}\label{CoroAgmon}
Let $\varepsilon >0$ and $\chi : M \rightarrow [0,1]$ be a smooth cutoff function, being 1 on a small neighborhood of $$K_{\eta} + \varepsilon = \lbrace x;\quad  \dd(x, K_{\eta}) < \varepsilon \rbrace.$$ Then, for any eigenpair $(\lambda,\psi)$ of $\Delta^{L^p}$, with $\lambda \leq (b_0 + \eta) p$ we have:
$$\psi = \chi \psi + \grandO(e^{-\varepsilon p^{\alpha}})\Vert \psi \Vert,$$
and
$$\nabla^{L^p} (\chi \psi) = \nabla^{L^p} \psi + \grandO (p^{1/2} e^{-\varepsilon p^{\alpha}})\Vert \psi \Vert,$$
uniformly with respect to $(\lambda, \psi)$.
\end{corollary}

\begin{proof}
By Theorem \ref{Thm-Agmon}, we have:
\begin{equation}\label{eq01-CoroAgmon}
\Vert (1- \chi) \psi \Vert^2 \leq \int_{(K_{\eta}+ \varepsilon)^c} \vert \psi \vert^2 \dd q \leq \int_M e^{-2 \varepsilon p^{\alpha}} \vert e^{\dd (x,K_{\eta})p^{\alpha}} \psi \vert^2 \dd x \leq C e^{-2 \varepsilon p^{\alpha}} \Vert \psi \Vert^2,
\end{equation}
which gives the first estimates. Moreover, we have with $\Phi(x) = \dd (x, K_{\eta})$,
$$ \Vert e^{\Phi p^{\alpha}} \nabla^{L^p} \psi \Vert \leq \Vert \nabla^{L^p} (e^{\Phi p^{\alpha}} \psi ) \Vert + p^{\alpha} \Vert \dd \Phi e^{\Phi p^{\alpha}} \psi \Vert,$$
and using Agmon's formula \ref{eq-AgmonFormula} and Theorem \ref{Thm-Agmon}:
$$ \Vert \nabla^{L^p} ( e^{\Phi p^{\alpha}} \psi) \Vert^2 = \lambda \Vert e^{\Phi p^{\alpha}} \psi \Vert^2 + p^{2\alpha} \Vert \dd \Phi e^{\Phi p^{\alpha}} \psi \Vert^2 \leq C^2 p \Vert \psi \Vert^2.$$
Thus,
\begin{equation}\label{eq02-CoroAgmon}
\Vert e^{\Phi p^{\alpha}} \nabla^{L^p} \psi \Vert \leq C p^{1/2} \Vert \psi \Vert^2.
\end{equation}
We can use these Agmon estimates on $\nabla^{L^p} \psi$ to get our second result. 
\begin{equation}\label{eq03-CoroAgmon}
\Vert \nabla^{L^p} ((1- \chi) \psi )\Vert \leq \Vert (\nabla^{L^p} \chi) \psi \Vert + \Vert (1- \chi) \nabla^{L^p} \psi \Vert
\end{equation}
The first term is dominated by
\begin{equation}
\Vert (\nabla^{L^p} \chi)\psi \Vert \leq C \Vert (1-\overline{\chi}) \psi \Vert
\end{equation}
where $\overline{\chi}$ is a cutoff function such that $\overline{\chi} = 1$ on $K_{\eta} + \varepsilon$ and $\overline{\chi} = 0$ on $\supp (1-\chi)$. We can apply \eqref{eq01-CoroAgmon} to $\overline{\chi}$ to get:
\begin{equation}\label{eq04-CoroAgmon}
\Vert (\nabla^{L^p} \chi) \psi \leq C e^{-\varepsilon p^{\alpha}} \Vert \psi \Vert.
\end{equation}
The second term of \eqref{eq03-CoroAgmon} is dominated as in \eqref{eq01-CoroAgmon}, using \eqref{eq02-CoroAgmon}:
\begin{equation}\label{eq05-CoroAgmon}
\Vert (1-\chi) \nabla^{L^p} \psi \Vert \leq C p^{1/2} e^{-\varepsilon p^{\alpha}} \Vert \psi \Vert.
\end{equation}
Finally, \eqref{eq03-CoroAgmon} with \eqref{eq04-CoroAgmon} and \eqref{eq05-CoroAgmon} yields
\begin{equation*}
\Vert \nabla^{L^p} ((1- \chi) \psi) \Vert \leq C p^{1/2} e^{-\varepsilon p^{\alpha}} \Vert \psi \Vert.
\end{equation*}

\end{proof}

\subsection{Comparison of the spectrum of $\Delta^{L^p}$ and $\Lpj$}

We recall that the minimum $b_0$ of $b$ is reached at $x_1, \cdots, x_N$ in a non-degenerate way. For $\eta>0$ small enough, the compact set $K_{\eta} = \lbrace b(x) \leq b_0 + \eta \rbrace$ has $N$ disjoint connected components $K_{\eta}^{(j)}$ such that $x_j \in K_{\eta}^{(j)}$. We fix the value of $\eta$, and we take $U_j$ a neighborhood of $K_{\eta}^{(j)}$. For $\varepsilon >0$ sufficiently small, $K_{\eta}^{(j)} + 2 \varepsilon \subset U_j$.\\

We denote by $B_j$ the restriction of $B$ to $U_j$. $\Lpj$ is the Dirichlet realisation of $(\dd + i p A_j)^*(\dd + i p A_j)$, with $A_j \in \Omega^1(U_j,L)$ such that $B_j = \dd A_j$. It is the self adjoint operator associated to the quadratic form:
\begin{equation}
Q_j(u,v) = \int_{U_j} (\dd + ip A_j)u \overline{(\dd + i p A_j )v} \dd x, \quad \forall u, v \in \mathsf{H}^1_0(U_j).
\end{equation}

Let us denote by 
\begin{equation}
K_p = \min \left[ N(\Delta^{L^p}, (b_0 + \eta)p) ; N \left( \bigoplus_{j=1}^N \Lpj , (0, b_0 + \eta)p \right) \right].
\end{equation}

We split the proof of Theorem \ref{MainTheorem} into two Lemmas.

\begin{lemma}
Let $\alpha \in (0,1/2)$. We have:
$$\lambda_k( \bigoplus_{j=1}^N \Lpj) \leq \lambda_k(\Delta^{L^p}) + \grandO(\exp(-\varepsilon p^{\alpha})),$$
uniformily with respect to $k \in [1, K_p]$.
\end{lemma}

\begin{proof}
We prove this using the min-max principle. For $k \in [1,J_p]$, let $\psi_k$ be the normalized eigenfunction associated to $\lambda_k(\Delta^{L^p})$. We will define the quasimode $u_{j,k} \in \mathcal{C}^{\infty}_0(U_j)$ using a local trivialisation of $L^p$ on $U_j$. Let $e_j \in \Cinf(U_j,L)$ be the non-vanishing local section of $L$ such that, for any $u \in \Cinf(U_j)$,
\begin{equation}\label{eq01-Lemma02}
\nabla^{L^p} (u e_j) = \left[(\dd + ip A_j)u \right] e_j.
\end{equation}
Let $\chi_j \in \Cinf_0(U_j)$ be a smooth cutoff function, such that $\chi_j = 1$ on $K_{\eta}^{(j)} + \varepsilon$. We define $u_{j,k} \in \Cinf_0(U_j)$ by $\chi_j \psi_k = u_{j,k} e_j$, and
$$u_k = u_{1,k} \oplus ... \oplus u_{N,k}.$$
Then
\begin{align*}
\langle \bigoplus_{j} \Lpj u_k, u_k \rangle = \sum_{j=1}^N \langle \Lpj u_{j,k}, u_{j,k} \rangle = \sum_{j=1}^N \Vert (\dd + i p A_j)u_{j,k} \Vert^2.
\end{align*}
Moreover, by \eqref{eq01-Lemma02},
$$\Vert (\dd + ip A_j) u_{j,k} \Vert^2 = \int_{U_j} \vert (\dd + i p A_j) u_{j,k} \vert^2 \dd x = \int_{U_j} \vert \nabla^{L^p} (\chi_j \psi_k) \vert^2 \dd x.$$
Now, $\chi = \sum_{j=1}^N \chi_j$ satisfies the assumptions of Corollary \ref{CoroAgmon} (with $2 \varepsilon$ instead of $\varepsilon$). Thus,
$$\langle \bigoplus_j \Lpj u_k, u_k \rangle = \int_M \vert \nabla^{L^p} (\chi \psi_k) \vert^2 \dd x = \Vert \nabla^{L^p} \psi_k \Vert^2 + \grandO(p^{1/2} e^{-2 \varepsilon p^{\alpha}})\Vert \psi_k \Vert,$$
uniformly with respect to $k$. $\psi_k$ being the eigensection associated to $\lambda_k(\Delta^{L^p})$, it remains:
$$\langle \bigoplus_j \Lpj u_k, u_k \rangle = \left( \lambda_k(\Delta^{L^p}) + \grandO (p{1/2} e^{-2\varepsilon p^{\alpha}}) \right)\Vert \psi_k \Vert.$$
This is true for every $k \in [1,K_p]$. Hence, for $1 \leq i \leq k \leq K_p$ we have
$$\langle \bigoplus_j \Lpj u_i, u_i \rangle \leq \left( \lambda_k(\Delta^{L^p}) + \grandO (p^{1/2}e^{-2\varepsilon p^{\alpha}}) \right)\Vert \psi_k \Vert,$$
and the Lemma follows from the min-max principle, because the vector space ranged by $(u_i)_{1 \leq i \leq k}$ is $k$-dimensional (and $p^{1/2} e^{-2 \varepsilon p^{\alpha}} = \grandO( e^{-\varepsilon p^{\alpha}})$).
\end{proof}

The reverse inequality is proven similarily.

\begin{lemma}
Let $\alpha \in (0,1/2)$. We have:
$$\lambda_k( \Delta^{L^p}) \leq \lambda_k(\bigoplus_{j=1}^N \Lpj) + \grandO(\exp(-\varepsilon p^{\alpha})),$$
uniformily with respect to $k \in [1, K_p]$.
\end{lemma}

\begin{proof}
The $k$-th eigenvalue of $\bigoplus_{j=1}^N \Lpj$ is given by an eigenpair $(\mu_{k},u_k)$ of $\mathcal{L}_p^{(j_k)}$ for some $j_k \in \lbrace 1, \cdots, N \rbrace$. Let $\chi_k \in \Cinf_0(U_{j_k})$ be a cutoff function equal to $1$ on $K_{\eta}^{(j_k)}+2\varepsilon$. Then, Agmon estimates (Theorem \ref{Thm-Agmon}) for $\Lpj$ imply that
\begin{align*}
(\dd + i p A) u_k = (\dd + i p A) (\chi_k u_k) + \grandO(e^{-\varepsilon p^{\alpha}})\Vert u_k \Vert
\end{align*}
uniformly with respect to $k$. We define $s_k = \chi_k u_k e_{j_k}$, where $e_{j_k}$ satisfies \eqref{eq01-Lemma02}, and we extend $s_k$ by $0$ outside $U_{j_k}$. Then,
$$
\begin{matrix*}[l]
\langle \Delta^{L^p} s_k, s_k \rangle &= \int_{U_{j_k}} \vert ( \dd + i p A )\chi_k u_k \vert^2 \dd x \\ &= \int_{U_{j_k}} \vert ( \dd + i p A ) u_k \vert^2 \dd x + \grandO(e^{-\varepsilon p^{\alpha}}) \\ &= \mu_k \Vert u_k \Vert^2 + \grandO(e^{- \varepsilon p^{\alpha}}).
\end{matrix*}$$
Hence the min-max principle implies 
$$\lambda_k(\Delta^{L^p}) \leq \mu_k + \grandO(e^{- \varepsilon p^{\alpha}}),$$
which is the desired inequality.
\end{proof}

\printbibliography

\end{document}